\documentclass{amsart}

\usepackage{amsthm,amsmath,amssymb,amsfonts}
\usepackage[all]{xy}
\usepackage{hyperref}
\usepackage{graphicx}
\usepackage{amsaddr}

\title{Simplification for Graph-like Objects}
\author{Will Grilliette}
\address{National Security Agency, Fort George G Meade MD, MD 20755-6844, USA}

\newtheorem{thm}{Theorem}[section]
\newtheorem{prop}[thm]{Proposition}
\newtheorem{cor}[thm]{Corollary}
\newtheorem{lem}[thm]{Lemma}

\theoremstyle{definition}
\newtheorem{defn}[thm]{Definition}

\theoremstyle{remark}

\newcommand{\set}{\mathbf{Set}}
\newcommand{\ssys}{\mathbf{SSys}}
\newcommand{\digra}{\mathbf{Digra}}
\newcommand{\istr}{\mathbf{IStr}}

\newcommand{\cat}[1]{\mathfrak{#1}}
\newcommand{\str}{\mathbf{Str}}
\newcommand{\spa}{\mathbf{Spa}}

\DeclareMathOperator{\ob}{Ob}
\DeclareMathOperator{\Space}{Sp}
\DeclareMathOperator{\comma}{Com}
\DeclareMathOperator{\Ran}{ran}

\begin{document}

\begin{abstract}
The simplification of a multigraph into a simple graph can be abstracted to a more general comma category under some common conditions.
When using the identity functor,
the category of simple objects in a comma category generalizes the functor-structured category.
Seated in categorical terms,
simplification can be dualized to ``antisimplification'',
which manifests as removal of isolated vertices and loose edges.
\end{abstract}

\maketitle

\section{Introduction}

In graph theory,
the process of converting a multigraph to a simple graph is mundane to the point of becoming virtually invisible in practice.
Due to this ubiquity,
one might not be surprised that
the simplification operation arises naturally as a left adjoint to the natural inclusion of simple graphs into the larger category of multigraphs.
Consequently,
the category of simple graphs inherits much of its parent category's structure.
Please note that this paper allows loops in simple graphs.

However,
in abstracting the simplification of a multigraph,
more connections become apparent.
The notions of ``complete graph'',
``simple graph'',
and ``simplification'' can be generalized within a comma category under some common conditions.
Explicitly,
following \cite{grilliette2023},
the codomain functor of the comma category must admit a right adjoint,
which serves as the analogue of the complete graph.
Then,
a ``simple object'' is a subobject of an image of this adjoint.
Simplification can,
therefore,
be represented as an image factorization of the unit of the adjunction.

When using the category of sets and its identity functor,
these conditions are invariably satisfied and produce the categories of simple digraphs and hypergraphs,
which are extant in the literature \cite{hajiabolhassan2016,hammack2016,hell1979-2}.
This special case of using the identity functor is not unique to the category of sets.
Indeed,
constructing a comma category with the identity functor and considering the category of simple objects creates a generalization of the functor-structured category \cite{adamek1980,kucera1972,solovyov2013}.

Yet,
with simplification seated in terms of a right adjoint and an image factorization,
a dual process of ``antisimplification'' can be considered.
An ``anticomplete'' object would be the image under a left adjoint to the domain functor,
and an ``antisimple'' object would be a quotient of an anticomplete object.
Then,
``antisimplification'' would result from a coimage factorization of the counit of the adjunction.

For directed multigraphs,
``anticomplete'' corresponds to a quiver composed of disjoint copies of the path of length 1,
and ``antisimple'' would be quotients of such,
which exclude isolated vertices.
Hence,
``antisimplification'' results in deletion of isolated vertices.
In this categorical sense,
simplification of a multigraph is dual to deletion of isolated vertices.
The ability to dualize graph-theoretic concepts in this way illustrates that category theory might provide novel perspectives on existing notions,
as well as new avenues of investigation.

Section \ref{simple} introduces the notions of ``complete'' and ``simple'' objects in a comma category,
before proving some general features of each.
The simplification is then defined on a comma category and characterized as a left adjoint to the natural inclusion.

Section \ref{structured} applies the situation of Section \ref{simple} to when one of the functors is the identity.
In this case,
a generalization of the functor-structured category manifests and behaves quite similarly to the category of simple graphs.

Section \ref{examples1} considers concrete examples of the framework introduced in Sections \ref{simple} and \ref{structured}.
This general framework is shown to capture both the transition of a set-system hypergraph to a simple set system and the transition of a quiver to a simple digraph.
In the case of incidence hypergraphs,
the simplification process collapses parallel incidences,
rather than parallel edges.
Lastly,
the simplification can be applied to the slice category of a regular category over a fixed object,
producing the ordered class of subobjects of said object.
The final example exemplifies how the image factorization appears to be key to this process.

Section \ref{examples2} demonstrates examples of antisimplification.
Applied to quivers,
antisimplification removes isolated vertices.
On the other hand,
antisimplification of an incidence hypergraph removes both isolated vertices and loose edges.
These two examples illustrate the connection between vertices of a quiver and an incidence hypergraph's vertices and edges found in the composition functor from \cite[Theorem 3.4.7]{grilliette2023}.

\section{Simplification}\label{simple}

Throughout this section,
let $\xymatrix{\cat{A}\ar[r]^{F} & \cat{C} & \cat{B}\ar[l]_{G}}$ be functors,
$\cat{G}:=\left(F\downarrow G\right)$ be the comma category,
$\xymatrix{\cat{A} & \cat{G}\ar[l]_{P}\ar[r]^{Q} & \cat{B}}$ be the canonical projections,
and $\xymatrix{FP\ar[r]^{f} & GQ}$ be the canonical natural transformation.
Note that $P$, $Q$, and $f$ generalize the edge set, vertex set, and incidence function of a graph, respectively \cite[p.\ 6-7]{grilliette2023}.

Recall that a simple graph is explicitly a subgraph of a complete graph \cite[p.\ 316]{kung1998}.
In \cite[p.\ 5-7]{grilliette2023},
the complete graph corresponded to a right adjoint to the vertex functor.
As such,
this section assumes that $Q$ admits a right adjoint $Q^{\star}$ with unit $\xymatrix{id_{\cat{G}}\ar[r]^{\eta} & Q^{\star}Q}$ and counit $\xymatrix{QQ^{\star}\ar[r]^{\theta} & id_{\cat{B}}}$,
and defines a ``complete object'' and a ``simple object'' in a comma category.

\begin{defn}[Generalized complete \& simple]
Let $H\in\ob(\cat{G})$.
\begin{enumerate}
\item $H$ is \emph{complete} if there is $B\in\ob(\cat{B})$ such that $H\cong Q^{\star}(B)$.
\item $H$ is \emph{simple} if there is a monomorphism $\xymatrix{H\textrm{ }\ar@{>->}[r]^(0.4){m} & Q^{\star}(B)}\in\cat{G}$ for some $B\in\ob(\cat{B})$.
\end{enumerate}
Let $\cat{SG}$ be the full subcategory of $\cat{G}$,
consisting of all simple objects in $\cat{G}$,
with inclusion functor $\xymatrix{\cat{SG}\ar[r]^{N} & \cat{G}}$.
A routine check shows that $\cat{SG}$ is replete in $\cat{G}$.
\end{defn}

Since the unit $\eta$ of the adjunction maps any object of $\cat{G}$ to a complete object,
it can be used to detect both completeness and simplicity.

\begin{prop}[Simple \& unit map]\label{simple:unit}
Let $H\in\ob(\cat{G})$.
\begin{enumerate}

\item\label{simple0} If $\eta_{H}$ is an isomorphism, then $H$ is complete.

\item If $Q^{\star}$ is full, then $H$ is complete if and only if $\eta_{H}$ is an isomorphism.

\item\label{simple1} $H$ is simple if and only if $\eta_{H}$ is monic.

\item\label{simple2} If $P\left(\eta_{H}\right)$ is monic,
then $H$ is simple.

\item If $P$ preserves monomorphisms,
then $H$ is simple if and only if $P\left(\eta_{H}\right)$ is monic.

\end{enumerate}
\end{prop}

\begin{proof}

Let $H\in\ob(\cat{G})$.
\begin{enumerate}

\item If $\eta_{H}$ is an isomorphism,
then $H\cong Q^{\star}Q(H)$.

\item $(\Leftarrow)$ If $\eta_{H}$ is an isomorphism,
then $H$ is complete by part \ref{simple0}.\\
$(\Rightarrow)$ Let $\xymatrix{H\ar[r]^(0.4){\alpha} & Q^{\star}(B)}\in\cat{G}$ be an isomorphism for some $B\in\ob(\cat{B})$.
As $Q\dashv Q^{\star}$,
$Q^{\star}\left(\theta_{B}\right)\circ\eta_{Q^{\star}(B)}=id_{Q^{\star}(B)}$,
so
\[
\left(\alpha^{-1}\circ Q^{\star}\left(\theta_{B}\right)\circ Q^{\star}Q(\alpha)\right)\circ\eta_{H}
=\alpha^{-1}\circ Q^{\star}\left(\theta_{B}\right)\circ\eta_{Q^{\star}(B)}\circ\alpha
=\alpha^{-1}\circ\alpha\\
=id_{H}.
\]
Since $Q^{\star}$ is full,
$\theta_{B}$ is a section,
meaning there is $\xymatrix{B\ar[r]^(0.4){\beta} & QQ^{\star}(B)}\in\cat{B}$ such that $\beta\circ\theta_{B}=id_{QQ^{\star}(B)}$.
Then,
\[
Q^{\star}(\beta)
=Q^{\star}(\beta)\circ Q^{\star}\left(\theta_{B}\right)\circ\eta_{Q^{\star}(B)}
=\eta_{Q^{\star}(B)},
\]
guaranteeing
\[\begin{array}{rcl}
\eta_{H}\circ\left(\alpha^{-1}\circ Q^{\star}\left(\theta_{B}\right)\circ Q^{\star}Q(\alpha)\right)
&   =   &   Q^{\star}Q\left(\alpha^{-1}\right)\circ\eta_{Q^{\star}(B)}\circ Q^{\star}\left(\theta_{B}\right)\circ Q^{\star}Q(\alpha)\\
&   =   &   Q^{\star}Q\left(\alpha^{-1}\right)\circ Q^{\star}Q(\alpha)
=id_{Q^{\star}Q(H)}.
\end{array}\]
Thus,
$\eta_{H}$ is an isomorphism.

\item $(\Leftarrow)$ If $\xymatrix{H\ar[r]^(0.4){\eta_{H}} & Q^{\star}Q(H)}\in\cat{G}$ is monic,
then $H$ is simple.\\
$(\Rightarrow)$ If $H$ is simple,
there is a monomorphism $\xymatrix{H\textrm{ }\ar@{>->}[r]^(0.4){m} & Q^{\star}(B)}\in\cat{G}$ for some $B\in\ob(\cat{B})$.
There is a unique $\xymatrix{Q(H)\ar[r]^(0.6){\hat{m}} & B}\in\cat{B}$ such that $Q^{\star}\left(\hat{m}\right)\circ\eta_{H}=m$.
Therefore,
$\eta_{H}$ is monic.

\item Observe that $\theta_{Q(H)}\circ Q\left(\eta_{H}\right)=id_{Q(H)}$ as $Q\dashv Q^{\star}$.
Thus,
$Q\left(\eta_{H}\right)$ is a section.
As both $Q\left(\eta_{H}\right)$ and $P\left(\eta_{H}\right)$ are monic,
$\eta_{H}$ is monic,
showing $H$ is simple by part \ref{simple1}.

\item $(\Leftarrow)$ If $P\left(\eta_{H}\right)$ is monic, then $H$ is simple by part \ref{simple2}.\\
$(\Rightarrow)$ If $H$ is simple,
then $\eta_{H}$ is monic by part \ref{simple1}.
As $P$ preserves monomorphisms,
$P\left(\eta_{H}\right)$ is monic.

\end{enumerate}

\end{proof}

Homomorphisms of simple graphs are often regarded as only functions between the vertex sets \cite[p.\ 53]{hell1979},
rather than a pair of functions \cite[p.\ 186]{dorfler1980},
due to the edge map being uniquely determined by the vertex map.
The same behavior can be observed in $\cat{SG}$,
due to the universal property of $Q^{\star}$.

\begin{prop}[Maps into simple]\label{simple:morphisms}
If $H$ is simple and $\xymatrix{K\ar@/^/[r]^{\phi}\ar@/_/[r]_{\psi} & H}\in\cat{G}$,
then $\phi=\psi$ if and only if $Q(\phi)=Q(\psi)$.
Thus,
morphisms in $\cat{SG}$ are uniquely determined by their $\cat{B}$-coordinate.
\end{prop}

\begin{proof}

$(\Rightarrow)$ This case is true in general.\\
$(\Leftarrow)$ Notice that
\[\begin{array}{rcl}
\theta_{Q(H)}\circ Q\left(\eta_{H}\circ\phi\right)
=Q\left(\phi\right)
=Q\left(\psi\right)
=\theta_{Q(H)}\circ Q\left(\eta_{H}\circ\psi\right).
\end{array}\]
By the universal property of $Q^{\star}Q(H)$,
$\eta_{H}\circ\phi=\eta_{H}\circ\psi$.
By Proposition \ref{simple:unit},
$\eta_{H}$ is monic,
so $\phi=\psi$.

\end{proof}

Simplifying a graph can be achieved through a quotient process \cite[p.\ 7]{bumby1984},
though this process is tacitly the image factorization of the incidence map.
Analogously,
if $\eta$ has an image factorization,
an object in $\cat{G}$ can be quotiented to a simple object.
Doing so for every object reveals $\cat{SG}$ to be a reflective subcategory of $\cat{G}$.

\begin{defn}[Simplification]
For $H\in\ob(\cat{G})$,
say $\eta_{H}$ admits a regular epi-mono factorization $\xymatrix{H\ar@{->>}[r]^{e_{H}} & S(H)\textrm{ }\ar@{>->}[r]^{m_{H}} & Q^{\star}Q(H)}\in\cat{G}$.
Observe that $S(H)\in\ob(\cat{SG})$,
and $\xymatrix{H\ar[r]^{e_{H}} & NS(H)}\in\cat{G}$
\end{defn}

\begin{thm}[Universal property]
For $H\in\ob(\cat{G})$,
let $\eta_{H}=m_{H}\circ e_{H}$ be a regular epi-mono factorization.
If $\xymatrix{H\ar[r]^(0.4){\phi} & N\left(H'\right)}\in\cat{G}$,
then there is a unique $\xymatrix{S(H)\ar[r]^(0.6){\hat{\phi}} & H'}\in\cat{SG}$ such that $N\left(\hat{\phi}\right)\circ e_{H}=\phi$.
\end{thm}

\begin{proof}

Let $\xymatrix{J\ar@/^/[r]^{p_{1}}\ar@/_/[r]_{p_{2}} & H}\in\cat{G}$ satisfy that $\xymatrix{H\ar@{->>}[r]^{e_{H}} & S(H)}$ is a coequalizer of $p_{1}$ and $p_{2}$.
By Proposition \ref{simple:unit},
$\eta_{H'}$ is monic.
Consider the diagram below.
\[\xymatrix{
J
\ar@/^/[r]^{p_{1}}
\ar@/_/[r]_{p_{2}}
&
H
\ar[rr]^{\phi}
\ar[d]_{\eta_{H}}
\ar@{}[drr]|-(0.6){=}
&
&
H'
\ar[d]^{\eta_{H'}}\\
&
Q^{\star}Q(H)\ar[rr]_{Q^{\star}Q(\phi)}
&
&
Q^{\star}Q\left(H'\right)
}\]
Then,
\[\begin{array}{rcl}
\eta_{H'}\circ\phi\circ p_{1}
&   =   &   Q^{\star}Q(\phi)\circ\eta_{H}\circ p_{1}\\
&   =   &   Q^{\star}Q(\phi)\circ m_{H}\circ e_{H}\circ p_{1}\\
&   =   &   Q^{\star}Q(\phi)\circ m_{H}\circ e_{H}\circ p_{2}\\
&   =   &   Q^{\star}Q(\phi)\circ\eta_{H}\circ p_{2}\\
&   =   &   \eta_{H'}\circ\phi\circ p_{2}.
\end{array}\]
As $\eta_{H'}$ is monic,
$\phi\circ p_{1}=\phi\circ p_{2}$.
Since $\xymatrix{H\ar@{->>}[r]^{e_{H}} & S(H)}$ is a coequalizer of $p_{1}$ and $p_{2}$,
there is a unique $\xymatrix{S(H)\ar[r]^(0.6){\hat{\phi}} & H'}\in\cat{G}$ such that $\hat{\phi}\circ e_{H}=\phi$.
Because $\cat{SG}$ is full in $\cat{G}$,
$\xymatrix{S(H)\ar[r]^(0.6){\hat{\phi}} & H'}\in\cat{SG}$ and $N\left(\hat{\phi}\right)\circ e_{H}=\phi$.

\end{proof}

\begin{cor}[Reflective subcategory]\label{simple:reflective}
Assume that $Q$ admits a right adjoint $Q^{\star}$ with unit $\xymatrix{id_{\cat{G}}\ar[r]^{\eta} & Q^{\star}Q}$,
and that $\eta_{H}$ admits a regular epi-mono factorization for all $H\in\ob(\cat{G})$.
Then,
$\cat{SG}$ is a reflective subcategory of $\cat{G}$.
\end{cor}

\section{Subobject-structured Categories}\label{structured}

For a functor $\xymatrix{\cat{B}\ar[r]^{G} & \cat{C}}$,
let $\cat{G}:=\left(id_{\cat{C}}\downarrow G\right)$,
$\xymatrix{\cat{C} & \cat{G}\ar[l]_{P}\ar[r]^{Q} & \cat{B}}$ be the canonical projections,
$\xymatrix{P\ar[r]^{f} & GQ}$ be the canonical natural transformation,
and $\cat{SG}$ be the full subcategory of simple objects in $\cat{G}$.
By \cite[Proposition 1.1.2]{grilliette2023},
$Q$ admits a right adjoint $Q^{\star}$ given on objects by
\[
Q^{\star}(B)=\left(G(B),id_{G(B)},B\right)
\]
with counit $\xymatrix{QQ^{\star}\ar[r]^{\theta} & id_{\cat{B}}}$ determined by $\theta_{B}=id_{B}$.
In this case,
the unit of the adjunction becomes intimately connected with $f$.

\begin{prop}[Unit map]
The unit $\xymatrix{id_{\cat{G}}\ar[r]^{\eta} & Q^{\star}Q}$ is given by
\[
\eta_{H}=\left(f_{H},id_{Q(H)}\right).
\]
\end{prop}

\begin{proof}

As $Q\dashv Q^{\star}$, one has
\[
Q\left(\eta_{H}\right)
=id_{Q(H)}\circ Q\left(\eta_{H}\right)
=\theta_{Q(H)}\circ Q\left(\eta_{H}\right)
=id_{Q(H)}.
\]
Since $\xymatrix{H\ar[r]^(0.4){\eta_{H}} & Q^{\star}Q(H)}\in\cat{G}$,
the following calculation follows:
\[
P\left(\eta_{H}\right)
=id_{GQ(H)}\circ P\left(\eta_{H}\right)
=f_{Q^{\star}Q(H)}\circ P\left(\eta_{H}\right)
=Q\left(\eta_{H}\right)\circ f_{H}
=id_{Q(H)}\circ f_{H}
=f_{H}.
\]

\end{proof}

Recall that a graph is \emph{simple} if it has no parallel edges \cite[p.\ 2-3, 31-32]{bondy-murty},
i.e.\ if the incidence function is one-to-one, or monic.
The proposition above demonstrates that a monic condition on $f$ implies simplicity,
but the two conditions actually coincide by exploiting the structure of $\eta$.

\begin{cor}[Simple \& incidence]\label{structured:unit}
An object $H\in\ob(\cat{G})$ is simple if and only if $f_{H}$ is monic.\end{cor}

\begin{proof}

$(\Leftarrow)$ If $f_{H}=P\left(\eta_{H}\right)$ is monic, then $H$ is simple by Proposition \ref{simple:unit}.\\
$(\Rightarrow)$ By Proposition \ref{simple:unit},
$\eta_{H}$ is monic.
Say $\xymatrix{C\ar@/^/[r]^(0.4){c_{1}}\ar@/_/[r]_(0.4){c_{2}} & P(H)}\in\cat{C}$ satisfy that $f_{H}\circ c_{1}=f_{H}\circ c_{2}$.
Define $h:=f_{H}\circ c_{1}$,
$H':=\left(C,h,Q(H)\right)\in\ob(\cat{G})$,
and $\phi_{k}:=\left(c_{k},id_{Q(H)}\right)$ for $k=1,2$.
Notice that $f_{H}\circ c_{k}=h\circ G\left(id_{Q(H)}\right)$,
meaning $\xymatrix{H'\ar@/^/[r]^{\phi_{1}}\ar@/_/[r]_{\phi_{2}} & H}\in\cat{G}$ for $k=1,2$.
Moreover,
\[
\eta_{H}\circ\phi_{1}
=\left(f_{H}\circ c_{1},id_{Q(H)}\circ id_{Q(H)}\right)
=\left(f_{H}\circ c_{2},id_{Q(H)}\circ id_{Q(H)}\right)
=\eta_{H}\circ\phi_{2}.
\]
As $\eta_{H}$ is monic,
$\phi_{1}=\phi_{2}$,
meaning $c_{1}=P\left(\phi_{1}\right)=P\left(\phi_{2}\right)=c_{2}$.
Therefore,
$f_{H}$ is monic.

\end{proof}

The connection with monomorphisms in $\cat{C}$ motivates a generalization of the functor-structured category \cite{adamek1980,kucera1972,solovyov2013}.

\begin{defn}[Subobject-structured]
A \emph{$G$-structured object} is a pair $(B,U)$,
where $B\in\ob(\cat{B})$ and $U$ is a subobject of $G(B)$.
For $G$-structured objects $\left(B_{k},U_{k}\right)$ for $k=1,2$,
a \emph{$G$-structured morphism} from $\left(B_{1},U_{1}\right)$ to $\left(B_{2},U_{2}\right)$ is a homomorphism $\xymatrix{B_{1}\ar[r]^{b} & B_{2}}\in\cat{B}$ such that there are
\[
\xymatrix{
G\left(B_{1}\right)    &   C_{1}\ar[r]^{c}\ar[l]_(0.4){m_{1}}   &   C_{2}\ar[r]^(0.4){m_{2}}  &   G\left(B_{2}\right)
}\in\cat{C}
\]
such that $\left(C_{k},m_{k}\right)$ represents the subobject $U_{k}$ for $k=1,2$
and $m_{2}\circ c=G(b)\circ m_{1}$.
\end{defn}

While the definition above merely asserts the existence of a trio of morphisms supporting $b$,
they are uniquely determined up to choice of the monomorphisms representing the subobjects.

\begin{prop}[Essential uniqueness]\label{structured:uniqueness}
Say $\xymatrix{B_{1}\ar[r]^{b} & B_{2}}\in\cat{B}$ is a $G$-structured morphism from $\left(B_{1},U_{1}\right)$ to $\left(B_{2},U_{2}\right)$.
If $\left(C_{k},m_{k}\right)$ represents the subobject $U_{k}$ for $k=1,2$,
then there is a unique $\xymatrix{C_{1}\ar[r]^{c} & C_{2}}\in\cat{C}$ such that $m_{2}\circ c=G(b)\circ m_{1}$.
\end{prop}

\begin{proof}

As $b$ is a $G$-structured morphism,
there are
\[
\xymatrix{
G\left(B_{1}\right)    &   D_{1}\ar[r]^{d}\ar[l]_(0.4){n_{1}}   &   D_{2}\ar[r]^(0.4){n_{2}}  &   G\left(B_{2}\right)
}\in\cat{C}
\]
such that $\left(D_{k},n_{k}\right)$ represents the subobject $U_{k}$ for $k=1,2$
and $n_{2}\circ d=G(b)\circ n_{1}$.
There is an isomorphism $\xymatrix{C_{k}\ar[r]^{\alpha_{k}}_{\cong} & D_{k}}\in\cat{C}$ such that $n_{k}\circ\alpha_{k}=m_{k}$ for $k=1,2$.
Let $c:=\alpha_{2}^{-1}\circ d\circ\alpha_{1}$,
and observe that
\[
m_{2}\circ c
=m_{2}\circ\alpha_{2}^{-1}\circ d\circ\alpha_{1}
=n_{2}\circ d\circ\alpha_{1}
=G(b)\circ n_{1}\circ\alpha_{1}
=G(b)\circ m_{1}.
\]
If $\xymatrix{C_{1}\ar[r]^{c'} & C_{2}}\in\cat{C}$ satisfies that $m_{2}\circ c'=G(b)\circ m_{1}$,
then $m_{2}\circ c'=m_{2}\circ c$.
Since $m_{2}$ is monic,
$c'=c$.

\end{proof}

As one would expect,
$G$-structured objects with $G$-structured morphisms form a category.

\begin{lem}[Composition \& identities]
If $\left(B_{k},U_{k}\right)$ is a $G$-structured object for $k=1,2,3$
and $\xymatrix{B_{k}\ar[r]^{b_{k}} & B_{k+1}}\in\cat{B}$ is a $G$-structured morphism for $k=1,2$,
then $b_{2}\circ b_{1}$ is a $G$-structured morphism from $\left(B_{1},U_{1}\right)$ to $\left(B_{3},U_{3}\right)$.
Moreover,
$id_{B_{1}}$ is a $G$-structured morphism from $\left(B_{1},U_{1}\right)$ to itself.
\end{lem}

\begin{proof}

There are$\xymatrix{G\left(B_{1}\right) & C_{1}\ar[r]^{c}\ar[l]_(0.4){m_{1}} & C_{2}\ar[r]^(0.4){m_{2}} & G\left(B_{2}\right)}\in\cat{C}$
such that $\left(C_{k},m_{k}\right)$ represents the subobject $U_{k}$ for $k=1,2$
and $m_{2}\circ c=G\left(b_{1}\right)\circ m_{1}$.
Likewise,
there are$\xymatrix{G\left(B_{2}\right) & D_{2}\ar[r]^{d}\ar[l]_(0.4){n_{2}} & D_{3}\ar[r]^(0.4){n_{3}} & G\left(B_{3}\right)}\in\cat{C}$
such that $\left(D_{k},n_{k}\right)$ represents the subobject $U_{k}$ for $k=2,3$
and $n_{3}\circ d=G\left(b_{2}\right)\circ n_{2}$.
There is an isomorphism $\xymatrix{C_{2}\ar[r]^{\alpha}_{\cong} & D_{2}}\in\cat{C}$ such that $n_{2}\circ\alpha=m_{2}$.
Then,
\[
n_{3}\circ\left(d\circ\alpha\circ c\right)
=G\left(b_{2}\right)\circ n_{2}\circ\alpha\circ c
=G\left(b_{2}\right)\circ m_{2}\circ c
=G\left(b_{2}\circ b_{1}\right)\circ m_{1}.
\]

Moreover,
$
m_{1}\circ id_{C_{1}}
=m_{1}
=G\left(id_{B_{1}}\right)\circ m_{1}
$.

\end{proof}

\begin{defn}[Category]
Define $\str(G)$ in the following way:
\begin{itemize}

\item $\ob(\str(G))$ is the class of $G$-structured objects;

\item for $(B,U),\left(B',U'\right)\in\ob(\str(G))$,
$\str(G)\left((B,U),\left(B',U'\right)\right)$ is the set of all $G$-structured morphisms from $(B,U)$ to $\left(B',U'\right)$;

\item identities and composition are inherited from $\cat{B}$.

\end{itemize}
By the previous lemma,
$\str(G)$ is a category.
\end{defn}

The content of Corollary \ref{structured:unit} and Proposition \ref{structured:uniqueness} shows that $\cat{SG}$ and $\str(G)$ are equivalent categories.

\begin{defn}[Functors]
Define $\xymatrix{\cat{SG}\ar[r]^(0.4){\Space} & \str(G)}$ in the following way:
\begin{itemize}

\item $\Space(H):=\left(Q(H),U_{H}\right)$,
where $U_{H}$ is the subobject of $GQ(H)$ represented by $\xymatrix{P(H)\textrm{ }\ar@{>->}[r]^{f_{H}} & GQ(H)}\in\cat{C}$;

\item $\Space(\phi):=Q(\phi)$.

\end{itemize}
Define $\xymatrix{\str(G)\ar[r]^{\comma} & \cat{SG}}$ in the following way:\begin{itemize}

\item for each $(B,U)\in\ob(\str(G))$,
choose $\xymatrix{C_{(B,U)}\textrm{ }\ar@{>->}[rr]^{m_{(B,U)}} & & G(B)}\in\cat{C}$ that represents $U$
and define $\comma(B,U):=\left(C_{(B,U)},m_{(B,U)},B\right)$;

\item for $\xymatrix{(B,U)\ar[r]^{b} & \left(B',U'\right)}\in\str(G)$,
Proposition \ref{structured:uniqueness} provides a unique
\[
\xymatrix{
C_{(B,U)}\ar[r]^{c_{b}} & C_{\left(B',U'\right)}
}\in\cat{C}
\]
such that $m_{\left(B',U'\right)}\circ c_{b}=G(b)\circ m_{(B,U)}$,
so define $\comma(b):=\left(c_{b},b\right)$.

\end{itemize}
Routine checks show that $\comma$ and $\Space$ are functors
and $\Space\comma=id_{\str(G)}$.
\end{defn}

\begin{thm}[Equivalence]\label{structured:equivalence}
There is a natural isomorphism $\xi$ from $id_{\cat{SG}}$ to $\comma\Space$.
Consequently,
$\cat{SG}$ is equivalent to $\str(G)$ as categories.
\end{thm}

\begin{proof}

For $H\in\ob(\cat{SG})$,
observe that
\[
\comma\Space(H)
=\comma\left(Q(H),U_{H}\right)
=\left(C_{\left(Q(H),U_{H}\right)},m_{\left(Q(H),U_{H}\right)},Q(H)\right),
\]
where $U_{H}$ is the subobject of $GQ(H)$ represented by $\xymatrix{P(H)\textrm{ }\ar@{>->}[r]^{f_{H}} & GQ(H)}\in\cat{C}$,
and $\xymatrix{C_{\left(Q(H),U_{H}\right)}\ar[rr]^{m_{\left(Q(H),U_{H}\right)}} & & GQ(H)}\in\cat{C}$ represents $U_{H}$.
Then,
there is an isomorphism $\xymatrix{P(H)\ar[r]^{\alpha_{H}} & C_{\left(Q(H),U_{H}\right)}}\in\cat{C}$ such that
\[
m_{\left(Q(H),U_{H}\right)}\circ\alpha_{H}
=f_{H}
=G\left(id_{Q(H)}\right)\circ f_{H}.
\]
Define $\xi_{H}:=\left(\alpha_{H},id_{Q(H)}\right)$
and note that $\xymatrix{H\ar[r]^(0.4){\xi_{H}} & \comma\Space(H)}\in\cat{SG}$ is an isomorphism.
Let $\xymatrix{H\ar[r]^{\phi} & H'}\in\cat{SG}$
and consider the diagram below.
\[\xymatrix{
H\ar[d]_{\phi}\ar[rr]^{\xi_{H}}   &   &   \comma\Space(H)\ar[d]^{\comma\Space(\phi)}\\
H'\ar[rr]_{\xi_{H'}}   &   &   \comma\Space\left(H'\right)\\
}\]
Then,
\[\begin{array}{rcl}
Q\left(\comma\Space(\phi)\circ\xi_{H}\right)
=Q(\phi)\circ id_{Q(H)}
=Q(\phi)
=id_{Q\left(H'\right)}\circ Q(\phi)
=Q\left(\xi_{H'}\circ\phi\right),
\end{array}\]
so $\xi_{H'}\circ\phi=\comma\Space(\phi)\circ\xi_{H}$ by Proposition \ref{simple:morphisms}.

\end{proof}

If $\cat{C}=\set$,
let $\spa(G)$ be the category of $G$-spaces and $G$-maps,
the functor-structured category \cite[Definition 5.40]{joyofcats}.
The categories $\str(G)$ and $\spa(G)$ coincide by recalling the correspondence between subsets and subobjects in $\set$.

\begin{prop}[Agreement of $\str$ and $\spa$]\label{structured:isomorphism}
If $\xymatrix{\cat{B}\ar[r]^{G} & \set}$ is a functor,
then $\str(G)\cong\spa(G)$ as categories.
\end{prop}

\begin{proof}

If $(B,U)\in\ob(\str(G))$,
then there is a unique subset $\alpha_{(B,U)}$ of $G(B)$ such that the set-theoretic inclusion $\xymatrix{\alpha_{(B,U)}\ar[r]^{\iota_{(B,U)}} & G(B)}\in\set$ represents the subobject $U$ by \cite[Example 7.81.1]{joyofcats}.
If $\xymatrix{\left(B_{1},U_{1}\right)\ar[r]^{b} & \left(B_{2},U_{2}\right)}\in\str(G)$,
there is a unique $\xymatrix{\alpha_{\left(B_{1},U_{1}\right)}\ar[r]^{f} & \alpha_{\left(B_{2},U_{2}\right)}}\in\set$ such that $\iota_{\left(B_{2},U_{2}\right)}\circ f=G(b)\circ\iota_{\left(B_{1},U_{1}\right)}$ by Proposition \ref{structured:uniqueness}.
For $x\in\alpha_{\left(B_{1},U_{1}\right)}$,
one has
\[
G(b)(x)
=\left(G(b)\circ\iota_{\left(B_{1},U_{1}\right)}\right)(x)
=\left(\iota_{\left(B_{2},U_{2}\right)}\circ f\right)(x)
=f(x)
\in\alpha_{\left(B_{2},U_{2}\right)}.
\]
Hence,
$b$ is a $G$-map.

Conversely,
if $(B,\alpha)\in\ob(\spa(G))$,
let $U_{(B,\alpha)}$ be the subobject of $G(B)$ represented by the set-theoretic inclusion $\xymatrix{\alpha\ar[r]^{\kappa_{(B,\alpha)}} & G(B)}\in\set$.
If
\[
\xymatrix{\left(B_{1},\alpha_{1}\right)\ar[r]^{b} & \left(B_{2},\alpha_{2}\right)}\in\spa(G),
\]
define $f:\alpha_{1}\to\alpha_{2}$ by $f(x):=G(b)(x)$.
For $x\in\alpha_{1}$,
\[
\left(\kappa_{\left(B_{2},\alpha_{2}\right)}\circ f\right)(x)
=f(x)
=G(b)(x)
=\left(G(b)\circ\kappa_{\left(B_{2},\alpha_{2}\right)}\right)(x).
\]
Hence,
$b$ is a $G$-structured morphism.

Define the functors $\xymatrix{\str(G)\ar[r]^{Z} & \spa(G)\ar[r]^{W} & \str(G)}$ by
\begin{itemize}
\item $Z(B,U):=\left(B,\alpha_{(B,U)}\right)$, $Z(b):=b$,
\item $W(B,\alpha):=\left(B,U_{(B,\alpha)}\right)$, $W(b):=b$.
\end{itemize}
Immediate calculations show that $WZ=id_{\str(G)}$ and $ZW=id_{\spa(G)}$.

\end{proof}

\section{Examples of Simplification}\label{examples1}

This section considers concrete examples of the simplification of Section \ref{simple} and the subobject-structured categories of Section \ref{structured}.
Sections \ref{examples:ssys}, \ref{examples:digra}, and \ref{examples:istr} consider the motivating examples from graph theory:  set-system hypergraphs, quivers, and incidence hypergraphs, respectively.
In the cases of set-system hypergraphs and quivers,
the simplification process is precisely the classical notion of graph simplification.
For an incidence hypergraph,
the simplification results in an incidence structure by collapsing parallel incidences.
Section \ref{examples:slice} handles the case of a slice category of a regular category,
where the simplification yields the range of the morphism.

\subsection{Set-System Hypergraphs}\label{examples:ssys}

Let $\mathcal{P}:\set\to\set$ be the covariant power-set functor \cite[p.\ 13]{maclane}.
The category $\cat{H}:=\left(id_{\set}\downarrow\mathcal{P}\right)$ is the category of set-system hypergraphs as seen in \cite{dorfler1980,grilliette2023}.
For a set-system hypergraph $H$,
$V(H)$ and $E(H)$ are the vertex set and edge set, respectively,
and $\epsilon_{H}$ is the incidence function.
Let $\cat{SH}$ be the full subcategory of $\cat{H}$ consisting of simple objects,
and $\xymatrix{\cat{SH}\ar[r]^{N_{\cat{H}}} & \cat{H}}$ be the inclusion functor.
As $\cat{H}$ is regular and $V$ admits a right adjoint \cite[p.\ 7, 10]{grilliette2023},
Corollary \ref{simple:reflective} provides a left adjoint $S_{\cat{H}}$ to $N_{\cat{H}}$.

Let $\ssys:=\spa\left(\mathcal{P}\right)$,
which is the category of set systems with their homomorphisms from \cite[p.\ 53]{hell1979}.
Theorem \ref{structured:equivalence} gives the equivalence below.
\[\xymatrix{
\cat{SH}
\ar[rr]^(0.4){\Space_{\cat{H}}}
& &
\str\left(\mathcal{P}\right)
\ar[rr]^(0.6){\comma_{\cat{H}}}
& &
\cat{SH}
}\]
If $\xymatrix{\str\left(\mathcal{P}\right)\ar[rr]^{Z_{\ssys}} & & \ssys}$ is the isomorphism from Proposition \ref{structured:isomorphism},
then the following diagram results.

\[\xymatrix{
\ssys  &   &
\cat{H}\ar@/_/[d]_{S_{\cat{H}}}^{\dashv}\\
\str\left(\mathcal{P}\right)\ar@/^/[rr]^{\comma_{\cat{H}}}\ar[u]^{Z_{\ssys}}_{\cong}   &   &
\cat{SH}\ar@/_/[u]_{N_{\cat{H}}}\ar@/^/[ll]^{\Space_{\cat{H}}}_{\simeq}
}\]

From direct computation,
the through-map is explicitly the simplification of a set-system hypergraph $H$ to a mere set system:
\[
Z_{\ssys}\Space_{\cat{H}}S_{\cat{H}}(H)
=\left(
V(H),
\left\{
\epsilon_{H}(e)
:
e\in E(H)
\right\}
\right).
\]

Moreover,
the following result is an immediate consequence of the composition.

\begin{thm}[Set systems \& hypergraphs]
The category $\ssys$ is equivalent to the reflective subcategory of simple set-system hypergraphs in $\cat{H}$.
Moreover,
$\ssys$ is complete and cocomplete with limits performed by passing to $\cat{H}$ and then applying the simplification.
\end{thm}

\subsection{Quivers}\label{examples:digra}

Let $\cat{E}$ be the finite category drawn below.
\[\xymatrix{
1\ar@/^/[rr]^{s}\ar@/_/[rr]_{t}   &   &   0
}\]
The category $\cat{Q}=\set^{\cat{E}}$ is the category of directed multigraphs,
or quivers,
as seen in \cite[p.\ 2]{bumby1984}.
Let $\Delta:\set\to\set\times\set$ be the diagonal functor,
which admits a right adjoint $\Delta^{\star}:\set\times\set\to\set$ given by the cartesian product \cite[p.\ 85]{maclane}.
The category $\cat{Q}_{1}:=\left(id_{\set}\downarrow\Delta^{\star}\Delta\right)$ is the category of quivers as seen in \cite{bondy-murty,knauer}.
The isomorphism between the functor category $\cat{Q}$ and the comma category $\cat{Q}_{1}$ follows from the universal property of the product.
\[\xymatrix{
&   \vec{E}(Q)\ar[dl]_{\sigma_{Q}}\ar[dr]^{\tau_{Q}}\ar@{..>}[d]|-{\exists!\vec{\epsilon}_{Q}}\\
\vec{V}(Q)  &   \vec{V}(Q)\times\vec{V}(Q)\ar[l]^(0.6){\pi_{1}^{Q}}\ar[r]_(0.6){\pi_{2}^{Q}}  &   \vec{V}(Q)
}\]
In the diagram above for a quiver $Q$,
$\vec{V}(Q)$ and $\vec{E}(Q)$ are the vertex and edge sets, respectively,
which occur in both the functor category representation and comma category representation.
Also,
$\sigma_{Q}$ and $\tau_{Q}$ are the source and target maps, respectively, from the functor category representation,
and $\vec{\epsilon}_{Q}$ is the incidence function from the comma category representation.
To tie these representations together,
$\pi_{n}^{Q}$ is the coordinate projection for $n=1,2$.

\begin{defn}[Isomorphism]
Define $\xymatrix{\cat{Q}\ar[r]^{W_{\cat{Q}}} & \cat{Q}_{1}\ar[r]^{Z_{\cat{Q}}} & \cat{Q}}$ by
\begin{itemize}

\item $W_{\cat{Q}}(Q):=\left(\vec{E}(Q),\vec{\epsilon}_{Q},\vec{V}(Q)\right)$,
where $\vec{\epsilon}_{Q}(e):=\left(\sigma_{Q}(e),\tau_{Q}(e)\right)$;

\item $W_{\cat{Q}}(\phi):=\left(\vec{E}(\phi),\vec{V}(\phi)\right)$;

\item $Z_{\cat{Q}}(Q):=\left(\vec{V}(Q),\vec{E}(Q),\pi_{1}^{Q}\circ\vec{\epsilon}_{Q},\pi_{2}^{Q}\circ\vec{\epsilon}_{Q}\right)$;

\item $Z_{\cat{Q}}(\phi):=\left(\vec{V}(\phi),\vec{E}(\phi)\right)$.

\end{itemize}
Routine calculations show that $W_{\cat{Q}}Z_{\cat{Q}}=id_{\cat{Q}_{1}}$ and $Z_{\cat{Q}}W_{\cat{Q}}=id_{\cat{Q}}$.
\end{defn}

Let $\cat{SQ}$ be the full subcategory of $\cat{Q}_{1}$ consisting of simple objects,
and
\[\xymatrix{
\cat{SQ}
\ar[rr]^{N_{\cat{Q}}}
& &
\cat{Q}_{1}
}\]
be the inclusion functor.
As $\cat{Q}$ is a topos and $\vec{V}$ admits a right adjoint \cite[p.\ 5]{grilliette2023},
Corollary \ref{simple:reflective} provides a left adjoint $S_{\cat{Q}}$ to $N_{\cat{Q}}$.

Let $\digra:=\spa\left(\Delta^{\star}\Delta\right)$,
which is the category of digraphs, or relations, from \cite{joyofcats,hell1979-2}.
Theorem \ref{structured:equivalence} provides the equivalence $\xymatrix{\cat{SQ}\ar[r]^(0.4){\Space_{\cat{Q}}} & \str\left(\Delta^{\star}\Delta\right)\ar[r]^(0.6){\comma_{\cat{Q}}} & \cat{SQ}}$.
If $\xymatrix{\str\left(\Delta^{\star}\Delta\right)\ar[rr]^{Z_{\digra}} & & \digra}$ is the isomorphism from Proposition \ref{structured:isomorphism}
then the following diagram results.

\[\xymatrix{
\digra  &   &   &   &
\cat{Q}\ar[d]^{W_{\cat{Q}}}_{\cong}\\
\str\left(\Delta^{\star}\Delta\right)\ar@/^/[rr]^(0.55){\comma_{\cat{Q}}}\ar[u]^{Z_{\digra}}_{\cong}   &   &
\cat{SQ}\ar@/^/[rr]^{N_{\cat{Q}}}\ar@/^/[ll]^(0.45){\Space_{\cat{Q}}}_{\simeq}    &   &
\cat{Q}_{1}\ar@/^/[ll]^{S_{\cat{Q}}}_{\top}
}\]

From direct computation,
the through-map is explicitly the simplification of a quiver $Q$ to a mere digraph:
\[
Z_{\digra}\Space_{\cat{Q}}S_{\cat{Q}}W_{\cat{Q}}(Q)
=\left(
\vec{V}(Q),
\left\{
\left(\sigma_{Q}(e),\tau_{Q}(e)\right)
:
e\in\vec{E}(Q)
\right\}
\right).
\]

Moreover,
the following result is an immediate consequence of the composition.

\begin{thm}[Digraphs \& quivers]
The category $\digra$ is equivalent to the reflective subcategory of simple quivers in $\cat{Q}$.
Moreover,
$\digra$ is complete and cocomplete with limits performed by passing to $\cat{Q}$ and then applying the simplification.
\end{thm}

\subsection{Incidence Hypergraphs}\label{examples:istr}

Let $\cat{D}$ be the finite category drawn below.
\[\xymatrix{
0   &   &   2\ar[ll]_{y}\ar[rr]^{z}   &   &   1
}\]
The category $\cat{R}=\set^{\cat{D}}$ is the category of incidence hypergraphs from \cite{grilliette2022,grilliette2023}.
The category $\cat{R}_{1}:=\left(id_{\set}\downarrow\Delta^{\star}\right)$ is the category of incidence hypergraphs as seen in \cite{chen2018,rusnak2018}.
The isomorphism between the functor category $\cat{R}$ and the comma category $\cat{R}_{1}$ follows from the universal property of the product.
\[\xymatrix{
&   I(G)\ar[dl]_{\varsigma_{G}}\ar[dr]^{\omega_{G}}\ar@{..>}[d]|-{\exists!\iota_{G}}\\
\check{V}(G)  &   \check{V}(G)\times\check{E}(G)\ar[l]^(0.6){\pi_{1}^{G}}\ar[r]_(0.6){\pi_{2}^{G}}  &   \check{E}(G)
}\]
In the diagram above for an incidence hypergraph $G$,
$\check{V}(G)$, $\check{E}(G)$, and $I(G)$ are the vertex, edge, and incidence sets, respectively,
which occur in both the functor category representation and comma category representation.
Also,
$\varsigma_{G}$ and $\omega_{G}$ are the port and attachment maps, respectively, from the functor category representation,
and $\iota_{G}$ is the incidence function from the comma category representation.
To tie these representations together,
$\pi_{n}^{G}$ is the coordinate projection for $n=1,2$.

\begin{defn}[Isomorphism]
Define $\xymatrix{\cat{R}\ar[r]^{W_{\cat{R}}} & \cat{R}_{1}\ar[r]^{Z_{\cat{R}}} & \cat{R}}$ by
\begin{itemize}

\item $W_{\cat{R}}(G):=\left(I(G),\iota_{G},\left(\check{V}(G),\check{E}(G)\right)\right)$,
where $\iota_{G}(j):=\left(\varsigma_{G}(j),\omega_{G}(j)\right)$;

\item $W_{\cat{R}}(\phi):=\left(I(\phi),\left(\check{V}(\phi),\check{E}(\phi)\right)\right)$;

\item $Z_{\cat{R}}(G):=\left(\check{V}(G),\check{E}(G),I(G),\pi_{1}^{G}\circ\iota_{G},\pi_{2}^{G}\circ\iota_{G}\right)$;

\item $Z_{\cat{R}}(\phi):=\left(\check{V}(\phi),\check{E}(\phi),I(\phi)\right)$.

\end{itemize}
Routine calculations show that $W_{\cat{R}}Z_{\cat{R}}=id_{\cat{R}_{1}}$ and $Z_{\cat{R}}W_{\cat{R}}=id_{\cat{R}}$.
\end{defn}

Let $\xymatrix{\cat{R}_{1}\ar[r]^(0.4){R} & \set\times\set}$ be the codomain functor,
which returns an ordered pair consisting of the vertex set and edge set.
Let $\cat{SR}$ be the full subcategory of $\cat{R}_{1}$ consisting of simple objects,
and $\xymatrix{\cat{SR}\ar[r]^{N_{\cat{R}}} & \cat{R}_{1}}$ be the inclusion functor.
As $\cat{R}$ is a topos and $R$ admits a right adjoint \cite[Proposition 1.1.2]{grilliette2023},
Corollary \ref{simple:reflective} provides a left adjoint $S_{\cat{R}}$ to $N_{\cat{R}}$.

Let $\istr:=\spa\left(\Delta^{\star}\right)$,
which is the category of incidence structures from \cite{beth,bumby1984,dembowski}.
Theorem \ref{structured:equivalence} yields the equivalence $\xymatrix{\cat{SR}\ar[r]^(0.4){\Space_{\cat{R}}} & \str\left(\Delta^{\star}\right)\ar[r]^(0.6){\comma_{\cat{R}}} & \cat{SR}}$.
If
\[\xymatrix{
\str\left(\Delta^{\star}\right)
\ar[rr]^{Z_{\istr}}
&
&
\istr
}\]
is the isomorphism from Proposition \ref{structured:isomorphism},
then the following diagram results.
\[\xymatrix{
\istr  &   &   &   &
\cat{R}\ar[d]^{W_{\cat{R}}}_{\cong}\\
\str\left(\Delta^{\star}\right)\ar@/^/[rr]^(0.55){\comma_{\cat{R}}}\ar[u]^{Z_{\istr}}_{\cong}   &   &
\cat{SR}\ar@/^/[rr]^{N_{\cat{R}}}\ar@/^/[ll]^(0.45){\Space_{\cat{R}}}_{\simeq}    &   &
\cat{R}_{1}\ar@/^/[ll]^{S_{\cat{R}}}_{\top}
}\]
From direct computation,
the through-map is explicitly the simplification of an incidence hypergraph $G$ to a mere incidence structure:
\[
Z_{\istr}\Space_{\cat{R}}S_{\cat{R}}W_{\cat{R}}(G)
=\left(
\left(\check{V}(G),\check{E}(G)\right),
\left\{
\left(\varsigma_{G}(j),\omega_{G}(j)\right)
:
j\in I(G)
\right\}
\right).
\]

Moreover,
the following result is an immediate consequence of the composition.

\begin{thm}[Incidence structures \& hypergraphs]
The category $\istr$ is equivalent to the reflective subcategory of simple incidence hypergraphs in $\cat{R}$.
Moreover,
$\istr$ is complete and cocomplete with limits performed by passing to $\cat{R}$ and then applying the simplification.
\end{thm}

\subsection{Slice Category}\label{examples:slice}

Fix a regular category $\cat{C}$ and an object $Y\in\ob(\cat{C})$.
Let $\cat{1}$ be the discrete category with a single object $\ast$,
and let $\xymatrix{\cat{1}\ar[r]^{K} & \cat{C}}$ be the constant functor determined by the map $\ast\mapsto Y$ \cite[Exercise 3K.b]{joyofcats}.
The slice category $\cat{C}/Y$ is isomorphic to the comma category $\cat{G}:=\left(id_{\cat{C}}\downarrow K\right)$ by the isomorphism below.

\begin{defn}[{Isomorphism, \cite[p.\ 92]{borceux1}}]
Define $\xymatrix{\cat{G}\ar[r]^{Z_{\cat{G}}} & \cat{C}/Y\ar[r]^{W_{\cat{G}}} & \cat{G}}$ by
\begin{itemize}

\item $Z_{\cat{G}}(C,f,\ast):=(C,f)$, $W_{\cat{G}}(C,f):=(C,f,\ast)$,

\item $Z_{\cat{G}}\left(\phi,id_{\ast}\right):=\phi$, $W_{\cat{G}}(\phi):=\left(\phi,id_{\ast}\right)$.

\end{itemize}
Routine calculations show that $W_{\cat{G}}Z_{\cat{G}}=id_{\cat{G}}$ and $Z_{\cat{G}}W_{\cat{G}}=id_{\cat{C}/Y}$.
\end{defn}

Let $\cat{SG}$ be the full subcategory of $\cat{G}$ consisting of simple objects,
and
\[\xymatrix{
\cat{SG}
\ar[rr]^{N_{\cat{G}}} 
& &
\cat{G}
}\]
be the inclusion functor.
As $\cat{1}$ is terminal,
the codomain functor of $\cat{G}$ is merely the constant functor,
whose right adjoint is the constant functor to the terminal object $\left(Y,id_{Y}\right)$ of the category \cite[p.\ 88]{maclane}.
Since $\cat{C}$ is regular,
$\cat{C}/Y$ is also \cite[Example 2.4.9]{borceux2}.
By Corollary \ref{simple:reflective},
$N_{\cat{G}}$ admits a left adjoint $S_{\cat{G}}$.

Via the isomorphism $W_{\cat{G}}$,
$\cat{SG}$ is isomorphic to the full subcategory $\mathbf{Mono}_{\cat{C}}(Y)$ of $\cat{C}/Y$ of monomorphisms \cite[p.\ 133]{borceux1}.
Observe that $\str(K)$ is isomorphic to the ordered class $\mathbf{Sub}_{\cat{C}}(Y)$ of subobjects of $Y$,
considered as a category,
through the isomorphism below.

\begin{defn}[Isomorphism, part 2]
Define the functors
\[\xymatrix{
\str(K)
\ar[rr]^{Z_{\cat{G}}'}
& &
\mathbf{Sub}_{\cat{C}}(Y)
\ar[rr]^{W_{\cat{G}}'}
& &
\str(K)
}\]
in the following way:
\begin{itemize}

\item $Z_{\cat{G}}'(\ast,U):=U$, $W_{\cat{G}}'(U):=(\ast,U)$,

\item $Z_{\cat{G}}'\left(id_{\ast}\right):=\phi$, $W_{\cat{G}}'(\phi):=id_{\ast}$,
where $\phi$ is the inclusion of the subobject $U$ into $U'$.

\end{itemize}
Routine calculations show that $W_{\cat{G}}'Z_{\cat{G}}'=id_{\str(K)}$ and $Z_{\cat{G}}'W_{\cat{G}}'=id_{\mathbf{Sub}_{\cat{C}}(Y)}$.
\end{defn}

If $\xymatrix{\cat{SG}\ar[r]^(0.4){\Space_{\cat{G}}} & \str\left(K\right)\ar[r]^(0.6){\comma_{\cat{G}}} & \cat{SG}}$ is the equivalence from Theorem \ref{structured:equivalence},
then the following diagram results.
\[\xymatrix{
\mathbf{Sub}_{\cat{C}}(Y)  &   &    \mathbf{Mono}_{\cat{C}}(Y)\ar[d]|-{\cong}   &   &
\cat{C}/Y\ar[d]^{W_{\cat{G}}}_{\cong}\\
\str\left(K\right)\ar@/^/[rr]^(0.55){\comma_{\cat{G}}}\ar[u]^{Z_{\cat{G}}'}_{\cong}   &   &
\cat{SG}\ar@/^/[rr]^{N_{\cat{G}}}\ar@/^/[ll]^(0.45){\Space_{\cat{G}}}_{\simeq}    &   &
\cat{G}\ar@/^/[ll]^{S_{\cat{G}}}_{\top}
}\]
From direct computation,
the through-map yields precisely the image of the morphism:
\[\begin{array}{rcl}
Z_{\cat{G}}'\Space_{\cat{G}}S_{\cat{G}}W_{\cat{G}}(C,f)
&   =   &   Z_{\cat{G}}'\Space_{\cat{G}}S_{\cat{G}}(C,f,\ast)\\
&   =   &   Z_{\cat{G}}'\Space_{\cat{G}}(R,m,\ast)\\
&   =   &   Z_{\cat{G}}'(\ast,U)\\
&   =   &   U,
\end{array}\]
where $\xymatrix{C\ar@{->>}[r]^{e} & R\textrm{ }\ar@{>->}[r]^{m} & Y}\in\cat{C}$ is an image factorization of $f$,
and $U$ is the subobject of $Y$ represented by $(R,m)$.
Moreover,
the following result is an immediate consequence of the composition.

\begin{thm}[Subobjects in a regular category]
Let $Y$ be a fixed object in a regular category $\cat{C}$.
The ordered class $\mathbf{Sub}_{\cat{C}}(Y)$ is equivalent, as a category, to the reflective subcategory of monomorphisms in $\cat{C}/Y$.
\end{thm}

\section{Examples of Antisimplification}\label{examples2}

This section considers concrete examples of ``antisimplification'',
the categorical dual notion to simplification.
Indeed,
notice that the simplification process presented in Section \ref{simple} can be dualized:
a left adjoint to the domain functor,
an epi-regular mono factorization of the counit.
One could also dualize the subobject-structured categories from Section \ref{structured} to provide ``quotient-structured categories''.
Here,
antisimplification is applied to the cases from Section \ref{examples1}.

Sadly,
antisimplification cannot be applied to the set-system hypergraphs of Section \ref{examples:ssys}
since the edge/domain functor $E$ does not admit a left adjoint \cite[Lemma 2.2.17]{grilliette2023}.
On the other hand,
the antisimplification applied to a coslice category of a coregular category will be precisely the dual of the case in Section \ref{examples:slice}.
Hence,
these two cases will be omitted.

The cases for quivers and incidence hypergraphs, however, are more interesting.
Section \ref{examples:r2} shows that the antisimplification of an incidence hypergraph,
rather than removing parallel incidences as in Section \ref{examples:istr},
instead removes isolated vertices and loose edges.
Likewise,
Section \ref{examples:q2} exhibits antisimplification of a quiver as removing isolated vertices,
instead of removing parallel edges as in Section \ref{examples:digra}.
These two characterizations seem to imply some duality between parallelism and isolation within a graph.

\subsection{Incidence Hypergraphs}\label{examples:r2}

Through the isomorphism $\xymatrix{\cat{R}\ar@/^/[r]^{W_{\cat{R}}} & \cat{R}_{1}\ar@/^/[l]^{Z_{\cat{R}}}}$,
the incidence functor $\xymatrix{\cat{R}\ar[r]^{I} & \set}$ corresponds to the domain functor of the comma category $\cat{R}_{1}$.
By \cite[p.\ 21]{grilliette2023},
$I$ admits a left adjoint $I^{\diamond}$ given on objects by
\[
I^{\diamond}(X)=\left(X,X,X,id_{X},id_{X}\right)
\]
with unit $\xymatrix{id_{\set}\ar[r]^{\eta} & II^{\diamond}}$ determined by $\eta_{X}=id_{X}$.
The counit codifies the structure of an incidence hypergraph into a homomorphism.

\begin{lem}[Counit]
The counit $\xymatrix{I^{\diamond}I\ar[r]^{\theta} & id_{\cat{R}}}$ is given by
\[
\theta_{G}
=
\left(\varsigma_{G},\omega_{G},id_{I(G)}\right).
\]
\end{lem}

\begin{proof}

Calculating,
\begin{itemize}

\item $
I\left(\theta_{G}\right)
=I\left(\theta_{G}\right)\circ id_{I(G)}
=I\left(\theta_{G}\right)\circ\eta_{I(G)}
=id_{I(G)}
$,

\item $
\check{V}\left(\theta_{G}\right)
=\check{V}\left(\theta_{G}\right)\circ id_{I(G)}
=\check{V}\left(\theta_{G}\right)\circ\varsigma_{I^{\diamond}I(G)}
=\varsigma_{G}\circ I\left(\theta_{G}\right)
=\varsigma_{G}\circ id_{I(G)}
=\varsigma_{G}
$,

\item $
\check{E}\left(\theta_{G}\right)
=\check{E}\left(\theta_{G}\right)\circ id_{I(G)}
=\check{E}\left(\theta_{G}\right)\circ\omega_{I^{\diamond}I(G)}
=\omega_{G}\circ I\left(\theta_{G}\right)
=\omega_{G}\circ id_{I(G)}
=\omega_{G}
$.

\end{itemize}

\end{proof}

Let $\cat{TR}_{1}$ be the full subcategory of $\cat{R}_{1}$ consisting of all antisimple objects,
and $\xymatrix{\cat{TR}_{1}\ar[r]^{M_{\cat{R}}} & \cat{R}_{1}}$ be the inclusion functor.
As $\cat{R}$ is a topos,
Corollary \ref{simple:reflective} provides a right adjoint $T_{\cat{R}}$ to $M_{\cat{R}}$.
Restricting $Z_{\cat{R}}$ and $W_{\cat{R}}$ to $\cat{TR}_{1}$ associates $\cat{TR}_{1}$ isomorphically to the full subcategory $\cat{TR}$ of $\cat{R}$ of incidence hypergraphs $G$,
where $\theta_{G}$ is an epimorphism.
Together,
these functors yield the following diagram.
\[\xymatrix{
\cat{TR}    &   &   \cat{R}\ar[d]^{W_{\cat{R}}}_{\cong}\\
\cat{TR}_{1}\ar@/^/[rr]^{M_{\cat{R}}}\ar[u]^{Z_{\cat{R}}'}_{\cong}    &   &   \cat{R}_{1}\ar@/^/[ll]^{T_{\cat{R}}}_{\bot}
}\]
From direct calculation,
the through-map explicitly removes any isolated vertices or loose edges from an incidence hypergraph $G$.
\[
Z_{\cat{R}}'T_{\cat{R}}W_{\cat{R}}(G)
=\left(
\Ran\left(\varsigma_{G}\right),
\Ran\left(\omega_{G}\right),
I(G),
\left.\varsigma_{G}\right|^{\Ran\left(\varsigma_{G}\right)},
\left.\omega_{G}\right|^{\Ran\left(\omega_{G}\right)}
\right)
\]
Thus,
vertices that are not incident to an edge,
and edges that are not incident to a vertex,
are removed.

\subsection{Quivers}\label{examples:q2}

Via the isomorphism $\xymatrix{\cat{Q}\ar@/^/[r]^{W_{\cat{Q}}} & \cat{Q}_{1}\ar@/^/[l]^{Z_{\cat{Q}}}}$,
the edge functor $\xymatrix{\cat{Q}\ar[r]^{\vec{E}} & \set}$ corresponds to the domain functor of the comma category $\cat{Q}_{1}$.
By \cite[p.\ 5]{grilliette2023},
$\vec{E}$ admits a left adjoint $\vec{E}^{\diamond}$ given on objects by
\[
\vec{E}^{\diamond}(X)=\left(\{0,1\}\times X,X,\varpi_{0},\varpi_{1}\right),
\]
where $\varpi_{n}(e):=(n,e)$ are the canonical inclusions into the disjoint union.
The unit $\xymatrix{id_{\set}\ar[r]^{\eta} & \vec{E}\vec{E}^{\diamond}}$ determined by $\eta_{X}=id_{X}$.
The counit weds source and target functions through the disjoint union into a homomorphism.

\begin{lem}[Counit]
The counit $\xymatrix{\vec{E}^{\diamond}\vec{E}\ar[r]^{\theta} & id_{\cat{Q}}}$ is given by
\[
\theta_{Q}
=
\left(\vec{V}\left(\theta_{Q}\right),id_{\vec{E}(Q)}\right),
\]
where $\vec{V}\left(\theta_{Q}\right)(n,e)=\left\{\begin{array}{cc}
\sigma_{Q}(e),  &   n=0,\\
\tau_{Q}(e),  &   n=1.
\end{array}\right.
$
\end{lem}

\begin{proof}

Calculating,
\begin{itemize}

\item $
\vec{E}\left(\theta_{Q}\right)
=\vec{E}\left(\theta_{Q}\right)\circ id_{\vec{E}(Q)}
=\vec{E}\left(\theta_{Q}\right)\circ\eta_{\vec{E}(Q)}
=id_{\vec{E}(Q)}
$,

\item $
\vec{V}\left(\theta_{Q}\right)(0,e)
=\left(\vec{V}\left(\theta_{Q}\right)\circ\varpi_{0}\right)(e)
=\left(\vec{V}\left(\theta_{Q}\right)\circ\sigma_{\vec{E}^{\diamond}\vec{E}(Q)}\right)(e)\\
=\left(\sigma_{Q}\circ\vec{E}\left(\theta_{Q}\right)\right)(e)
=\left(\sigma_{Q}\circ id_{\vec{E}(Q)}\right)(e)
=\sigma_{Q}(e)
$,

\item $
\vec{V}\left(\theta_{Q}\right)(1,e)
=\left(\vec{V}\left(\theta_{Q}\right)\circ\varpi_{1}\right)(e)
=\left(\vec{V}\left(\theta_{Q}\right)\circ\tau_{\vec{E}^{\diamond}\vec{E}(Q)}\right)(e)\\
=\left(\tau_{Q}\circ\vec{E}\left(\theta_{Q}\right)\right)(e)
=\left(\tau_{Q}\circ id_{\vec{E}(Q)}\right)(e)
=\tau_{Q}(e)
$.

\end{itemize}

\end{proof}

Let $\cat{TQ}_{1}$ be the full subcategory of $\cat{Q}_{1}$ consisting of all antisimple objects,
and $\xymatrix{\cat{TQ}_{1}\ar[r]^{M_{\cat{Q}}} & \cat{Q}_{1}}$ be the inclusion functor.
As $\cat{Q}$ is a topos,
Corollary \ref{simple:reflective} provides a right adjoint $T_{\cat{Q}}$ to $M_{\cat{Q}}$.
Restricting $Z_{\cat{Q}}$ and $W_{\cat{Q}}$ to $\cat{TQ}_{1}$ associates $\cat{TQ}_{1}$ isomorphically to the full subcategory $\cat{TQ}$ of $\cat{Q}$ of quivers $Q$,
where $\theta_{Q}$ is an epimorphism.
Together,
these functors yield the following diagram.
\[\xymatrix{
\cat{TQ}    &   &   \cat{Q}\ar[d]^{W_{\cat{Q}}}_{\cong}\\
\cat{TQ}_{1}\ar@/^/[rr]^{M_{\cat{Q}}}\ar[u]^{Z_{\cat{Q}}'}_{\cong}    &   &   \cat{Q}_{1}\ar@/^/[ll]^{T_{\cat{Q}}}_{\bot}
}\]
From direct calculation,
the through-map explicitly removes any isolated vertices from a quiver $Q$.
\begin{center}\scalebox{0.95}{$
Z_{\cat{Q}}'T_{\cat{Q}}W_{\cat{Q}}(Q)
=\left(
\Ran\left(\sigma_{Q}\right)\cup\Ran\left(\tau_{Q}\right),
\vec{E}(Q),
\left.\sigma_{Q}\right|^{\Ran\left(\sigma_{Q}\right)\cup\Ran\left(\tau_{Q}\right)},
\left.\tau_{Q}\right|^{\Ran\left(\sigma_{Q}\right)\cup\Ran\left(\tau_{Q}\right)}
\right)
$}\end{center}
Thus,
vertices that are not the source or target of an edge are removed.

\bibliographystyle{plain}
\bibliography{main}

\begin{thebibliography}{10}

\bibitem{joyofcats}
Jir{\i} Ad{\'a}mek, Horst Herrlich, and George~E Strecker.
\newblock Abstract and concrete categories; the joy of cats 2004.
\newblock {\em Reprints in Theory and Applications of Categories}, 17:1--507,
  2006.

\bibitem{adamek1980}
Jiří Adámek and Václav Koubek.
\newblock Cartesian closed functor-structured categories.
\newblock {\em Commentationes Mathematicae Universitatis Carolinae},
  021(3):573--590, 1980.

\bibitem{beth}
Thomas Beth, D~Jungnickel, and Hanfried Lenz.
\newblock {\em Encyclopedia of mathematics and its applications design theory:
  Series number 69: Volume 1}.
\newblock Cambridge University Press, Cambridge, England, 2 edition, November
  1999.

\bibitem{bondy-murty}
J.A. Bondy and U.S.R Murty.
\newblock {\em Graph Theory}.
\newblock Springer Publishing Company, Incorporated, 1st edition, 2008.

\bibitem{borceux1}
Francis Borceux.
\newblock {\em Handbook of categorical algebra I: Basic category theory
  (encyclopedia of mathematics and its applications)}.
\newblock Cambridge University Press, 1994.

\bibitem{borceux2}
Francis Borceux.
\newblock {\em Encyclopedia of mathematics and its applications handbook of
  categorical algebra: Series number 51: Categories and structures volume 2}.
\newblock Cambridge University Press, Cambridge, England, April 2008.

\bibitem{bumby1984}
Richard~T. Bumby and Dana~May Latch.
\newblock Categorical constructions in graph theory.
\newblock {\em International Journal of Mathematics and Mathematical Sciences},
  9:791947, Jan 1984.

\bibitem{chen2018}
Gina Chen, Vivian Liu, Ellen Robinson, Lucas~J. Rusnak, and Kyle Wang.
\newblock A characterization of oriented hypergraphic laplacian and adjacency
  matrix coefficients.
\newblock {\em Linear Algebra and its Applications}, 556:323--341, 2018.

\bibitem{dembowski}
Peter Dembowski.
\newblock {\em Finite geometries}.
\newblock Classics in Mathematics. Springer, Berlin, Germany, December 1996.

\bibitem{dorfler1980}
W.~D{\"o}rfler and D.~A. Waller.
\newblock A category-theoretical approach to hypergraphs.
\newblock {\em Archiv der Mathematik}, 34(1):185--192, Dec 1980.

\bibitem{grilliette2022}
Will Grilliette, Josephine Reynes, and Lucas~J. Rusnak.
\newblock Incidence hypergraphs: Injectivity, uniformity, and matrix-tree
  theorems.
\newblock {\em Linear Algebra and its Applications}, 634:77--105, 2022.

\bibitem{grilliette2023}
Will Grilliette and Lucas~J. Rusnak.
\newblock Incidence hypergraphs: the categorical inconsistency of set-systems
  and a characterization of quiver exponentials.
\newblock {\em Journal of Algebraic Combinatorics}, 58(1):1--36, Aug 2023.

\bibitem{hajiabolhassan2016}
Hossein Hajiabolhassan and Frédéric Meunier.
\newblock Hedetniemi's conjecture for kneser hypergraphs.
\newblock {\em Journal of Combinatorial Theory, Series A}, 143:42--55, 2016.

\bibitem{hammack2016}
Richard~H. Hammack, Marc Hellmuth, Lydia Ostermeier, and Peter~F. Stadler.
\newblock Associativity and non-associativity of some hypergraph products.
\newblock {\em Mathematics in Computer Science}, 10(3):403--408, Sep 2016.

\bibitem{hell1979-2}
Pavol Hell.
\newblock An introduction to the category of graphs.
\newblock {\em Annals of the New York Academy of Sciences}, 328(1):120--136,
  1979.

\bibitem{hell1979}
Pavol Hell and Jaroslav Nešetřil.
\newblock Cohomomorphisms of graphs and hypergraphs.
\newblock {\em Mathematische Nachrichten}, 87(1):53--61, 1979.

\bibitem{knauer}
Ulrich Knauer.
\newblock {\em Algebraic graph theory morphisms, monoids, and matrices}.
\newblock De Gruyter, 2011.

\bibitem{kung1998}
Joseph~P.S. Kung.
\newblock A geometric condition for a hyperplane arrangement to be free.
\newblock {\em Advances in Mathematics}, 135(2):303--329, 1998.

\bibitem{kucera1972}
L.~Kučera and A.~Pultr.
\newblock On a mechanism of defining morphisms in concrete categories.
\newblock {\em Cahiers de Topologie et Géométrie Différentielle
  Catégoriques}, 13(4):397--410, 1972.

\bibitem{maclane}
Saunders~Mac Lane.
\newblock {\em Categories for the working mathematician}.
\newblock Graduate Texts in Mathematics. Springer, New York, NY, 2 edition,
  September 1998.

\bibitem{rusnak2018}
Lucas~J. Rusnak, Ellen Robinson, Martin Schmidt, and Piyush Shroff.
\newblock Oriented hypergraphic matrix-tree type theorems and bidirected minors
  via boolean order ideals.
\newblock {\em Journal of Algebraic Combinatorics}, 49(4):461--473, June 2018.

\bibitem{solovyov2013}
Sergey~A. Solovyov.
\newblock Categorically algebraic topology versus universal topology.
\newblock {\em Fuzzy Sets and Systems}, 227:25--45, 2013.
\newblock Theme: Topology.

\end{thebibliography}

\end{document}